\pdfoutput=1
\documentclass[11pt]{article}

\usepackage[T1]{fontenc}%
\usepackage{graphicx}%
\usepackage{multirow}%
\usepackage{amsmath,amssymb,amsfonts}%
\usepackage{amsthm}%
\usepackage{mathrsfs}%
\usepackage[mathscr]{eucal}%
\usepackage[title]{appendix}%
\usepackage{xcolor}%
\usepackage{textcomp}%
\usepackage{manyfoot}%
\usepackage{booktabs}%
\usepackage{authblk}%
\usepackage{microtype}%

\newtheoremstyle{thmstyleone}%
  {\topsep}
  {\topsep}
  {\itshape}
  {}
  {\bfseries}
  {.}
  {.5em}
  {}

\newtheoremstyle{thmstyletwo}%
  {\topsep}
  {\topsep}
  {}
  {}
  {\bfseries}
  {.}
  {.5em}
  {}

\newtheoremstyle{thmstylethree}%
  {\topsep}
  {\topsep}
  {}
  {}
  {\itshape}
  {.}
  {.5em}
  {}

\theoremstyle{thmstyleone}%
\newtheorem{theorem}{Theorem}%

\newtheorem{proposition}[theorem]{Proposition}%

\theoremstyle{thmstyletwo}%
\newtheorem{remark}{Remark}%

\theoremstyle{thmstylethree}%
\newtheorem{definition}{Definition}%
\newtheorem{corollary}{Corollary}%
\newtheorem{lemma}{Lemma}%

\title{Limit-Cycle Replication via Chebyshev Pullbacks and a Quadratic Ceiling for Separable Schemes}

\author[1]{Olimjon Eshkobilov}
\author[1]{Shirali Kadyrov}
\author[1]{Khudoyor Mamayusupov}

\affil[1]{Department of General Education, New Uzbekistan University, Movarounnahr 1, Tashkent 100007, Uzbekistan}

\date{}

\begin{document}

\maketitle

\begin{center}
\small
\texttt{o.eshqobilov@newuu.uz}\\
\texttt{sh.kadyrov@newuu.uz}\\
\texttt{k.mamayusupov@newuu.uz}\\
These authors contributed equally to this work.
\end{center}

\begin{abstract}
Let \(H(n)\) denote the Hilbert number, i.e.\ the maximal number of limit cycles of planar polynomial vector fields of degree \(\le n\).
A classical lower-bound mechanism for \(H(n)\) is \emph{replication}: one pulls back a vector field by a polynomial map and lifts each existing limit cycle to several disjoint copies while controlling the resulting degree.
In this paper we give a fully self-contained replication theorem based on the separable Chebyshev covering
\[
\Phi(u,v)=(T_m(u),T_m(v)).
\]
Using the \(m\) monotone full branches of \(T_m\) on \((-1,1)\), we prove that every degree-\(\le n\) polynomial vector field with \(k\) limit cycles gives rise to a degree-\(\le nm+m-1\) polynomial vector field with at least \(m^2k\) limit cycles.
Consequently,
\[
H(nm+m-1)\ge m^2H(n)\qquad (m\ge 2).
\]
We then extend the construction to general separable pullbacks \((u,v)\mapsto (p(u),p(v))\), show that Chebyshev attains the maximal possible branch count among degree-\(m\) separable pullbacks, and prove a quadratic ceiling for replication-only schemes: if one iterates separable pullbacks and no additional limit cycles are created beyond those forced by lifting, then the number of resulting limit cycles is at most quadratic in the final degree.
This shows that superquadratic lower bounds, such as the known \(n^2\log n\)-type bounds, necessarily require mechanisms beyond pure separable replication.
Finally, combining our replication theorem with the strongest currently published seed bounds, we obtain new explicit lower estimates in several degrees, including
\begin{gather*}
H(14)\ge 252,\qquad H(29)\ge 1080,\\
H(31)\ge 1380,\qquad H(39)\ge 2012.
\end{gather*}
\end{abstract}


\section{Introduction}\label{sec:introduction}

Planar polynomial differential systems
\begin{equation}\label{eq:polyvf}
\dot x = P(x,y),\qquad \dot y = Q(x,y),\qquad P,Q\in\mathbb{R}[x,y],
\end{equation}
form a central class in the qualitative theory of dynamical systems and bifurcation theory; see, for instance,
\cite{Perko2001,DumortierLlibreArtes2006}.
A persistent driving problem is to quantify how the global phase portrait can grow in complexity as the algebraic
complexity of 
\eqref{eq:polyvf} 
increases. The most celebrated incarnation is the second part of Hilbert's 16th problem,
which asks for uniform bounds on the number and relative disposition of limit cycles of polynomial vector fields of a
given degree (see, e.g., 
\cite{Ilyashenko2002,Roussarie1998,Gasull2013} 
for background and surveys).

A \emph{planar polynomial vector field} is the vector field $X=(P,Q)$ in 
\eqref{eq:polyvf}
, and its (total) \emph{degree} is
\[
\deg(X):=\max\{\deg P,\deg Q\}.
\]
A \emph{periodic orbit} is a non-equilibrium closed trajectory of the flow, and a \emph{limit cycle} is a periodic orbit
that is isolated among periodic orbits (equivalently, it corresponds to an isolated fixed point of a local Poincar\'e
return map on a transverse section). For each $n\ge 1$, the associated \emph{Hilbert number} is
\begin{equation*}
\begin{aligned}
H(n) &:= \sup\bigl\{\pi(X): X \text{ is a planar polynomial vector field with }\\
     &\deg(X)\le n\bigr\},
\end{aligned}
\end{equation*}
where $\pi(X)$ denotes the number of limit cycles of $X$ (possibly $+\infty$).
Determining whether $H(n)$ is finite for each $n$ remains widely open; in particular, the finiteness of $H(2)$ is not known
(see 
\cite{GasullSantanaPAMS2025}
and the discussion in 
\cite{Ilyashenko2002}).
A major line of research surrounding Hilbert's 16th problem concerns finiteness theorems and their quantitative refinements;
classical references include 
\cite{Ilyashenko1991,Roussarie1998}.
Smale presents the following contemporary formulation of the second part of Hilbert's 16th problem: given a polynomial differential system
\eqref{eq:polyvf}
in the plane, does there exist an upper bound \(K\) on the number of limit cycles that depends only on 
the degree \(d=\deg(X)\) of the polynomials, specifically of the form \(K \leq d^{q}\) for some universal constant \(q\) 
\cite{Smale1998}?

On the one hand, there is substantial progress on structural aspects conditional on finiteness.
For example, assuming $H(n)<\infty$, Gasull--Santana prove that $H(n)$ is realized by a structurally stable vector field
with hyperbolic limit cycles and establish a monotonicity step $H(n+1)\ge H(n)+1$; see 
\cite{GasullSantanaPAMS2025}.
On the other hand, lower-bound constructions show that $H(n)$ grows at least quadratically along infinite sequences of
degrees, and several constructions achieve asymptotics of order $n^2\log n$ via suitable amplification/replication steps
combined with cycle-creation mechanisms; see 
\cite{ChristopherLloyd1995,LiChanChung2002,HanLi2012,AlvarezCollMaesschalckProhens2020}
and the recent overview 
\cite{GasullSantanaCPAA2026}
(cf.\ also 
\cite{BuzziNovaes2024}
for discussion and context).
Related recent work includes explicit polynomial families exhibiting
many limit cycles in constrained geometries (e.g.\ an invariant square with at least five limit cycles in a quartic family)
and quantitative lower bounds in generalized Liénard-type settings; see, for instance, 
\cite{GrauSzantoQTD2024,AbreuMartinsQTD2024}.

A recurring methodological theme in these lower bounds is \emph{replication} (or \emph{amplification}): one pulls back a
vector field by a polynomial map and uses the covering structure to lift each existing limit cycle to many disjoint copies
at a controlled increase of degree. This mechanism appears explicitly in the ``quadruple transformation'' of
Christopher--Lloyd \cite{ChristopherLloyd1995} and subsequent refinements (e.g.\ \cite{LiChanChung2002,HanLi2012}),
and it often serves as a scaffold on which additional local bifurcations are arranged to create new cycles.

\subsection*{Main results}

Our first contribution is a fully self-contained replication step based on the separable Chebyshev covering
\[
(u,v)\longmapsto \bigl(T_m(u),T_m(v)\bigr).
\]
The Chebyshev polynomial $T_m$ has $m$ monotone full branches on $(-1,1)$ (Lemma~\ref{lem:chebbranches}); hence the
product map is a diffeomorphism on each of the $m^2$ branch rectangles and lifts each limit cycle to $m^2$ disjoint copies.
This yields the following amplification inequality.

\begin{theorem}[Chebyshev replication inequality]\label{thm:replication}
Let \(n\ge 1\) and \(m\ge 2\). Then
\[
H(nm+m-1)\ \ge\ m^2\,H(n).
\]
Equivalently: given any degree-\(\le n\) planar polynomial vector field having \(k\) limit cycles, one can construct a
degree-\(\le(nm+m-1)\) planar polynomial vector field having at least \(m^2k\) limit cycles.
\end{theorem}

\subsection*{Position relative to existing lower bounds}

Theorem~\ref{thm:replication} should be interpreted as an optimal result within the
class of \emph{separable replication schemes}, rather than as a replacement for the
stronger cycle-creation mechanisms developed in the Hilbert-number literature.
In particular, the lower-bound constructions of Christopher--Lloyd, Li--Chan--Chung,
and Han--Li combine replication with additional bifurcation arguments and thereby
produce superquadratic growth, of order \(N^{2}\log N\), for suitable infinite
families of degrees; see
\cite{ChristopherLloyd1995,LiChanChung2002,HanLi2012}.
By contrast, Theorem~\ref{thm:quadraticceiling} below shows that \emph{pure}
separable replication, under the assumption that no new cycles are created beyond
those forced by lifting, cannot by itself produce more than \(O(N^{2})\) limit cycles.
From this perspective, Theorem~\ref{thm:replication} identifies the optimal
multiplicative gain \(m^{2}\) available from a degree-\(m\) separable pullback and
pinpoints where an additional cycle-creation mechanism is needed in order to obtain
superquadratic growth.

A search of the literature up to the present shows that the explicit global
degree-by-degree lower bounds relevant here are still governed by Han--Li
\cite{HanLi2012} and Prohens--Torregrosa \cite{ProhensTorregrosa2019}; more recent
papers such as \cite{GasullSantanaPAMS2025,GasullSantanaCPAA2026} contribute
structural monotonicity and new recurrent viewpoints, but do not replace the
numerical records used in the comparison below.
Accordingly, for each target degree \(N\) we compare the best currently published
bound \(L_{\mathrm{pub}}(N)\) located in the literature with the strongest \emph{direct}
consequence of Theorem~\ref{thm:replication},
\[
L_{\mathrm{Ch}}(N)
:=
\max\bigl\{ m^{2}L_{\mathrm{pub}}(n)\;:\; N=(n+1)m-1,\ m\ge 2\bigr\}.
\]
Equivalently, if \(N+1=(n+1)m\), then Theorem~\ref{thm:replication} gives
\[
H(N)\ge m^{2}H(n)\ge m^{2}L_{\mathrm{pub}}(n).
\]

\begin{table}[htbp]
\centering
\small
\setlength{\tabcolsep}{5pt}
\begin{tabular}{ccccc}
\toprule
\(N\) & \(L_{\mathrm{pub}}(N)\) & \(L_{\mathrm{Ch}}(N)\) & seed \((n,m)\) & \(\Delta:=L_{\mathrm{Ch}}(N)-L_{\mathrm{pub}}(N)\)\\
\midrule
11 & \(\mathbf{153}\) \cite[Thm.~1.2(i)]{HanLi2012}                  & \(148\)             & \((5,2)\)  & \(-5\)   \\
13 & \(\mathbf{212}\) \cite[Cor.~2(a)]{ProhensTorregrosa2019}       & \(\mathbf{212}\)    & \((6,2)\)  & \(0\)    \\
14 & \(194\) \cite[Thm.~1.2(i)]{HanLi2012}                           & \(\mathbf{252}\)    & \((4,3)\)  & \(58\)   \\
15 & \(\mathbf{345}\) \cite[Thm.~1.2(i)]{HanLi2012}                  & \(296\)             & \((7,2)\)  & \(-49\)  \\
17 & \(\mathbf{384}\) \cite[Cor.~2(a)]{ProhensTorregrosa2019}       & \(\mathbf{384}\)    & \((8,2)\)  & \(0\)    \\
19 & \(\mathbf{503}\) \cite[Thm.~1.2(i)]{HanLi2012}                  & \(480\)             & \((9,2)\)  & \(-23\)  \\
20 & \(\mathbf{509}\) \cite[Thm.~1.2(i)]{HanLi2012}                  & \(477\)             & \((6,3)\)  & \(-32\)  \\
21 & \(\mathbf{568}\) \cite[Cor.~2(a)]{ProhensTorregrosa2019}       & \(\mathbf{568}\)    & \((10,2)\) & \(0\)    \\
23 & \(\mathbf{833}\) \cite[Thm.~1.2(i)]{HanLi2012}                  & \(666\)             & \((7,3)\)  & \(-167\) \\
24 & \(\mathbf{843}\) \cite[Thm.~1.2(i)]{HanLi2012}                  & \(700\)             & \((4,5)\)  & \(-143\) \\
25 & \(\mathbf{870}\) \cite[Thm.~1.2(i)]{HanLi2012}                  & \(628\)             & \((12,2)\) & \(-242\) \\
26 & \(\mathbf{880}\) \cite[Thm.~1.2(i)]{HanLi2012}                  & \(864\)             & \((8,3)\)  & \(-16\)  \\
27 & \(\mathbf{1023}\) \cite[Thm.~1.2(i)]{HanLi2012}                 & \(848\)             & \((13,2)\) & \(-175\) \\
29 & \(1060\) \cite[Thm.~1.2(i)]{HanLi2012}                          & \(\mathbf{1080}\)   & \((9,3)\)  & \(20\)   \\
31 & \(1184\) \cite[Cor.~2(a)]{ProhensTorregrosa2019}               & \(\mathbf{1380}\)   & \((15,2)\) & \(196\)  \\
35 & \(\mathbf{1536}\) \cite[Cor.~2(a)]{ProhensTorregrosa2019}      & \(\mathbf{1536}\)   & \((17,2)\) & \(0\)    \\
39 & \(1920\) \cite[Cor.~2(a)]{ProhensTorregrosa2019}               & \(\mathbf{2012}\)   & \((19,2)\) & \(92\)   \\
43 & \(\mathbf{2272}\) \cite[Cor.~2(a)]{ProhensTorregrosa2019}      & \(\mathbf{2272}\)   & \((21,2)\) & \(0\)    \\
\bottomrule
\end{tabular}
\caption{Comparison between the best published degree-specific lower bounds available in the
literature and the strongest direct consequence of
Theorem~\ref{thm:replication}. In each row, the better of the two bounds is set in bold;
when the two bounds coincide, both are shown in bold. The seed values used to compute
\(L_{\mathrm{Ch}}(N)\) are documented in Appendix~\ref{app:derivations}.}
\label{tab:pub-vs-cheb}
\end{table}

Table~\ref{tab:pub-vs-cheb} gives a more accurate picture than comparing only the three
degrees \(11\), \(14\), and \(29\).
First, it shows that once all currently available seed bounds are allowed, the strongest
direct degree-\(11\) consequence of Theorem~\ref{thm:replication} is not
\(9H(3)\ge 117\), but rather
\[
H(11)\ge 4H(5)\ge 148,
\]
which still remains slightly below Han--Li’s value \(153\).
Second, the theorem already \emph{matches} several current published records,
namely at \(N=13,17,21,35,43\).
This is not accidental: these rows come from the four-fold replication step
\(H(2n+1)\ge 4H(n)\), which is precisely the special case \(m=2\) of
Theorem~\ref{thm:replication}.
Third, and most importantly, the theorem \emph{improves} the currently published
values at
\[
H(14)\ge 252,\qquad
H(29)\ge 1080,\qquad
H(31)\ge 1380,\qquad
H(39)\ge 2012.
\]
Among these, the improvements at \(N=31\) and \(N=39\) seem especially worth recording:
they come from combining the \(m=2\) case of Theorem~\ref{thm:replication} with the
Han--Li seed bounds \(H(15)\ge 345\) and \(H(19)\ge 503\), respectively.

Accordingly, the role of Theorem~\ref{thm:replication} is twofold.
On the one hand, it gives a sharp and fully self-contained replication theorem for
separable polynomial pullbacks, with Chebyshev polynomials realizing the maximal branch
count.
On the other hand, when paired with the strongest seed systems currently available in the
literature, it yields several new competitive degree-specific lower bounds, even though
pure replication alone cannot explain the known \(O(N^{2}\log N)\)-type growth.

For a concrete illustration of Theorem~\ref{thm:replication},
Section~\ref{sec:workedexample} starts from an explicit cubic system with a single
hyperbolic limit cycle and applies the Chebyshev pullback with \(m=3\), producing \(9\)
pairwise disjoint hyperbolic limit cycles in the \(3\times 3\) branch rectangles at final
degree \(11\).

\begin{corollary}[Four-fold replication]
Taking \(m=2\) yields
\[
H(2n+1)\ge 4H(n).
\]
\end{corollary}

As noted in \cite{GasullSantanaCPAA2026}, this four-fold replication step can already be
extracted from the construction in \cite[\S~3]{ChristopherLloyd1995}: after affine
localization and a separable squaring-type change of variables, each limit cycle lifts to
one copy in each quadrant, yielding the factor \(4\).

\begin{remark}\label{rem:sharpness-replication}
The factor \(m^{2}\) in Theorem~\ref{thm:replication} cannot be improved within the class
of \emph{separable} pullbacks \((u,v)\mapsto (p(u),p(v))\) with \(\deg(p)=m\).
Indeed, any real polynomial \(p\) of degree \(m\) has at most \(m\) full monotone
branches mapping onto \((-1,1)\) (Lemma~\ref{lem:branchbound}), so a separable pullback
can lift a given limit cycle to at most \(m^{2}\) disjoint copies
(cf.\ Proposition~\ref{prop:generalrep}).
Chebyshev attains this bound by Lemma~\ref{lem:chebbranches}.
Moreover, for genuinely degree-\(n\) fields the degree growth in the pullback construction
is exact,
\[
\deg(Y)=m\,\deg(X)+(m-1)
\]
(Lemma~\ref{lem:exactdegree}),
so the target degree \(nm+m-1\) cannot, in general, be lowered without changing the
pullback mechanism.

On the other hand, Theorem~\ref{thm:replication} is not expected to be sharp for the
Hilbert numbers themselves, since known constructions create additional limit cycles
beyond those forced by lifting, leading in particular to superquadratic lower bounds;
see, for example, \cite{HanLi2012,AlvarezCollMaesschalckProhens2020}.
Nevertheless, the theorem can still be competitive at the degree-by-degree level when
combined with strong seed bounds.
For example, taking \((n,m)=(4,3)\) gives \(N=4\cdot 3+3-1=14\), and hence
\[
H(14)\ge 9H(4).
\]
Using the bound \(H(4)\ge 28\) from \cite{ProhensTorregrosa2019}, this yields
\[
H(14)\ge 252,
\]
which improves the value \(H(14)\ge 194\) recorded in \cite{HanLi2012}.
Thus the sharpness issue is twofold: the separable replication factor \(m^{2}\) is optimal
within its class, while the resulting Hilbert-number lower bounds may or may not be
globally optimal depending on the quality of the available seed systems.
\end{remark}

Our second contribution isolates a limitation of \emph{pure separable replication}.
Starting from a degree-\(n_{0}\) system with \(k_{0}\) limit cycles, consider any
construction obtained by iterating only separable polynomial pullbacks
\((u,v)\mapsto (p(u),p(v))\), with \(\deg(p)\ge 2\), under the standing assumption that no
limit cycles are created except those forced by lifting previously existing cycles through
full branches.
Then the total number of limit cycles in the resulting degree-\(N\) system is bounded by a
quadratic function of \(N\).

\begin{theorem}[Quadratic ceiling for pure separable replication]
\label{thm:quadraticceiling}
Fix \(n_{0}\ge 1\).
Let \(X_{0}\) be a planar polynomial vector field of degree \(\le n_{0}\) having
\(k_{0}\) limit cycles.
Let \(X\) be any planar polynomial vector field of degree \(N\) obtained from \(X_{0}\)
by a finite sequence of separable polynomial pullbacks
\((u,v)\mapsto (p(u),p(v))\) with \(\deg(p)\ge 2\), assuming that at each step
\emph{no} limit cycles are created beyond the lifted copies forced by the full-branch
structure of the pullback.
Then
\[
\pi(X)\le k_{0}\left(\frac{N+1}{n_{0}+1}\right)^{2}.
\]
In particular, any such replication-only scheme yields at most \(O(N^{2})\) limit cycles
as \(N\to\infty\).
\end{theorem}

Consequently, any superquadratic lower-bound strategy must incorporate mechanisms beyond
separable lifting alone, consistent with the structure of the sharpest existing
constructions; see, for example, \cite{HanLi2012,AlvarezCollMaesschalckProhens2020}.

\medskip\noindent\textbf{Organization.}
Section~\ref{sec:tools} records two elementary invariance facts: affine coordinate changes
and nonvanishing time reparametrizations.
Section~\ref{sec:chebyshev} recalls the monotone-branch structure of Chebyshev polynomials.
Section~\ref{sec:replication} proves Theorem~\ref{thm:replication}.
Section~\ref{sec:separable} abstracts the construction to general separable pullbacks and
establishes the quadratic ceiling in Theorem~\ref{thm:quadraticceiling}.
Finally, Section~\ref{sec:workedexample} provides an explicit illustrative example of
Theorem~\ref{thm:replication}.

\section{Two elementary tools}\label{sec:tools}

The replication argument uses two standard invariances: 
affine coordinate changes allow us to place all relevant limit
cycles inside a prescribed box, and multiplying a vector field by a nowhere-vanishing scalar does not change its orbit
set (hence preserves periodic orbits and limit cycles). 
We record both facts for later reference.

\begin{lemma}[Degree under affine changes]\label{lem:affine}
Let \(A:\mathbb{R}^2\to\mathbb{R}^2\) be an invertible affine map \(A(w)=Mw+b\) with \(M\in GL_2(\mathbb{R})\).
Let \(\dot z=X(z)\) be a planar polynomial vector field of degree \(\le n\).
Under the change of variables \(z=A(w)\), the system becomes
\[
\dot w = M^{-1}X(Mw+b),
\]
which is a polynomial vector field of degree \(\le n\).
\end{lemma}

\begin{proof}
Write \(z=A(w)=Mw+b\). Differentiating gives \(\dot z = M\dot w\), hence
\[
M\dot w = X(Mw+b)\qquad\Rightarrow\qquad \dot w = M^{-1}X(Mw+b).
\]
Since \(X\) is polynomial, the composition \(X(Mw+b)\) is polynomial and does not increase total degree, and
multiplication by the constant matrix \(M^{-1}\) does not change degree. Therefore the transformed vector field is
polynomial of degree \(\le n\).
\end{proof}

In the replication construction we will obtain an orbit correspondence of the form \(D\Phi\cdot Y=\lambda\,X\circ\Phi\)
with \(\lambda\neq 0\) on each branch. The following standard lemma explains why multiplying a vector field by a
nowhere-vanishing function does not change trajectories and preserves Poincar\'e return maps up to inversion.

\begin{lemma}[Nonvanishing time change preserves trajectories]\label{lem:timechange}
Let \(X\) be a \(C^1\) vector field on an open set \(\Omega\subset\mathbb{R}^2\), and 
let \(\lambda:\Omega\to\mathbb{R}\) be locally Lipschitz with \(\lambda(z)\neq 0\) for all \(z\in\Omega\).
Define \(\widetilde X(z):=\lambda(z)\,X(z)\).
Then:
\begin{enumerate}
\item \(X\) and \(\widetilde X\) have the same unparametrized trajectories in \(\Omega\) (with reversed orientation on connected regions where \(\lambda<0\)).
\item A periodic orbit of \(X\) in \(\Omega\) is a periodic orbit of \(\widetilde X\), and conversely.
\item If \(\gamma\) is a limit cycle of \(X\) in \(\Omega\), then \(\gamma\) is a limit cycle of \(\widetilde X\) in \(\Omega\).
\item Let \(\gamma\) be a periodic orbit and let \(\Sigma\) be a sufficiently small transverse section through a point of \(\gamma\).
Write \(P_X\) and \(P_{\widetilde X}\) for the corresponding local Poincar\'e return maps defined using forward time for each flow.
Then on their common domain, either \(P_{\widetilde X}=P_X\) (if \(\lambda>0\) in a neighborhood of \(\gamma\)) or \(P_{\widetilde X}=P_X^{-1}\) (if \(\lambda<0\) in a neighborhood of \(\gamma\)).
In particular, \(P_X\) and \(P_{\widetilde X}\) have the same fixed point set near \(\gamma\), and a fixed point is isolated for \(P_X\) iff it is isolated for \(P_{\widetilde X}\).
\end{enumerate}
\end{lemma}

\begin{proof}
Since \(\lambda\) is locally Lipschitz and nowhere zero, both ODEs
\(\dot z=X(z)\) and \(\dot z=\widetilde X(z)\) have unique maximal solutions through each initial condition.

\smallskip\noindent
\emph{(1) Orbit sets coincide.}
Let \(z(t)\) be a maximal solution of \(\dot z=X(z)\) with \(z(t_0)=z_0\).
Define
\[
\tau(t):=\int_{t_0}^{t}\frac{1}{\lambda(z(s))}\,ds.
\]
Because \(\lambda(z(s))\neq 0\), the function \(\tau(t)\) is \(C^1\) and strictly monotone on any interval on which \(z\) is defined (increasing if \(\lambda>0\) along the orbit segment, decreasing if \(\lambda<0\)).
Hence it has a local \(C^1\) inverse \(t=t(\tau)\) on compact subintervals of its image.
Set \(w(\tau):=z(t(\tau))\). By the chain rule,
\[
\frac{dw}{d\tau}
=\frac{dz}{dt}\frac{dt}{d\tau}
=X(w)\,\lambda(w)
=\widetilde X(w).
\]
Thus the same point set \(\{z(t)\}\) is traversed by a solution of \(\dot w=\widetilde X(w)\) (with reversed orientation when \(\tau\) decreases).
The converse implication (from \(\widetilde X\)-solutions to \(X\)-solutions) is obtained by the symmetric reparametrization
\(t(\tau)=\int_{\tau_0}^{\tau}\frac{1}{\lambda(w(\sigma))}\,d\sigma\).
Therefore \(X\) and \(\widetilde X\) have the same unparametrized trajectories, proving (1).

\smallskip\noindent
\emph{(2) Periodic orbits are preserved.}
If \(\gamma\) is a periodic orbit of \(X\), then it is a compact trajectory set of \(X\), hence by (1) it is also a trajectory set of \(\widetilde X\).
Since \(\gamma\) is closed and non-equilibrium, it is a periodic orbit for \(\widetilde X\). The converse is identical.

\smallskip\noindent
\emph{(4) Return maps: equality or inverse.}
Let \(\gamma\) be a periodic orbit and \(\Sigma\) a small transverse section through a point \(p\in\gamma\).
By continuity and compactness of \(\gamma\), the function \(\lambda\) has constant sign on a sufficiently small tubular neighborhood of \(\gamma\); 
shrink \(\Sigma\) so that the forward return map is defined for both flows on a common neighborhood of \(p\) in \(\Sigma\).

If \(\lambda>0\) near \(\gamma\), then the time-change reparametrization above preserves forward orientation along trajectories. 
Hence, for any \(q\in\Sigma\) close to \(p\), the next intersection of the \(X\)-orbit of \(q\) with \(\Sigma\) is the same point as the next intersection of the \(\widetilde X\)-orbit of \(q\) with \(\Sigma\). Thus \(P_{\widetilde X}=P_X\).

If \(\lambda<0\) near \(\gamma\), then the same orbit is traversed in the opposite direction. 
In that case, the \emph{forward} return for \(\widetilde X\) corresponds to the \emph{backward} return for \(X\), so \(P_{\widetilde X}=P_X^{-1}\) on the common domain.

In either case, \(P_X\) and \(P_{\widetilde X}\) have the same fixed point set near \(p\), and isolation of a fixed point is preserved under inversion, proving the last assertion in (4).

\smallskip\noindent
\emph{(3) Limit cycles are preserved.}
A periodic orbit \(\gamma\) is a limit cycle iff the corresponding fixed point of the local Poincar\'e map on a transverse section is isolated. 
By (4), isolation is preserved between \(P_X\) and \(P_{\widetilde X}\). Hence \(\gamma\) is a limit cycle for \(X\) iff it is a limit cycle for \(\widetilde X\).
\end{proof}

\section{Chebyshev polynomials as multibranch coverings}\label{sec:chebyshev}

\begin{definition}[Chebyshev polynomial]
For \(m\ge 1\), the Chebyshev polynomial of the first kind \(T_m\) is defined by
\[
T_m(\cos\theta)=\cos(m\theta)\quad(\theta\in\mathbb{R}).
\]
Then \(T_m\) is a real polynomial of degree \(m\) and \(T_m([-1,1])=[-1,1]\).
\end{definition}

\begin{lemma}[Monotone branches of \(T_m\)]\label{lem:chebbranches}
Fix \(m\ge 2\) and define
\[
c_k := \cos\Bigl(\frac{k\pi}{m}\Bigr),\qquad k=0,1,\dots,m.
\]
Then \(1=c_0>c_1>\cdots>c_m=-1\). On each open interval
\[
I_k := (c_{k},\,c_{k-1})\qquad (k=1,\dots,m),
\]
the map \(T_m\) is \(C^\infty\), strictly monotone, and satisfies \(T_m(I_k)=(-1,1)\). 
Hence \(T_m:I_k\to(-1,1)\) is a diffeomorphism and \(T_m'(x)\neq 0\) for all \(x\in I_k\).
\end{lemma}

\begin{figure}[htbp!]
  \centering
  \includegraphics[width=0.6\linewidth]{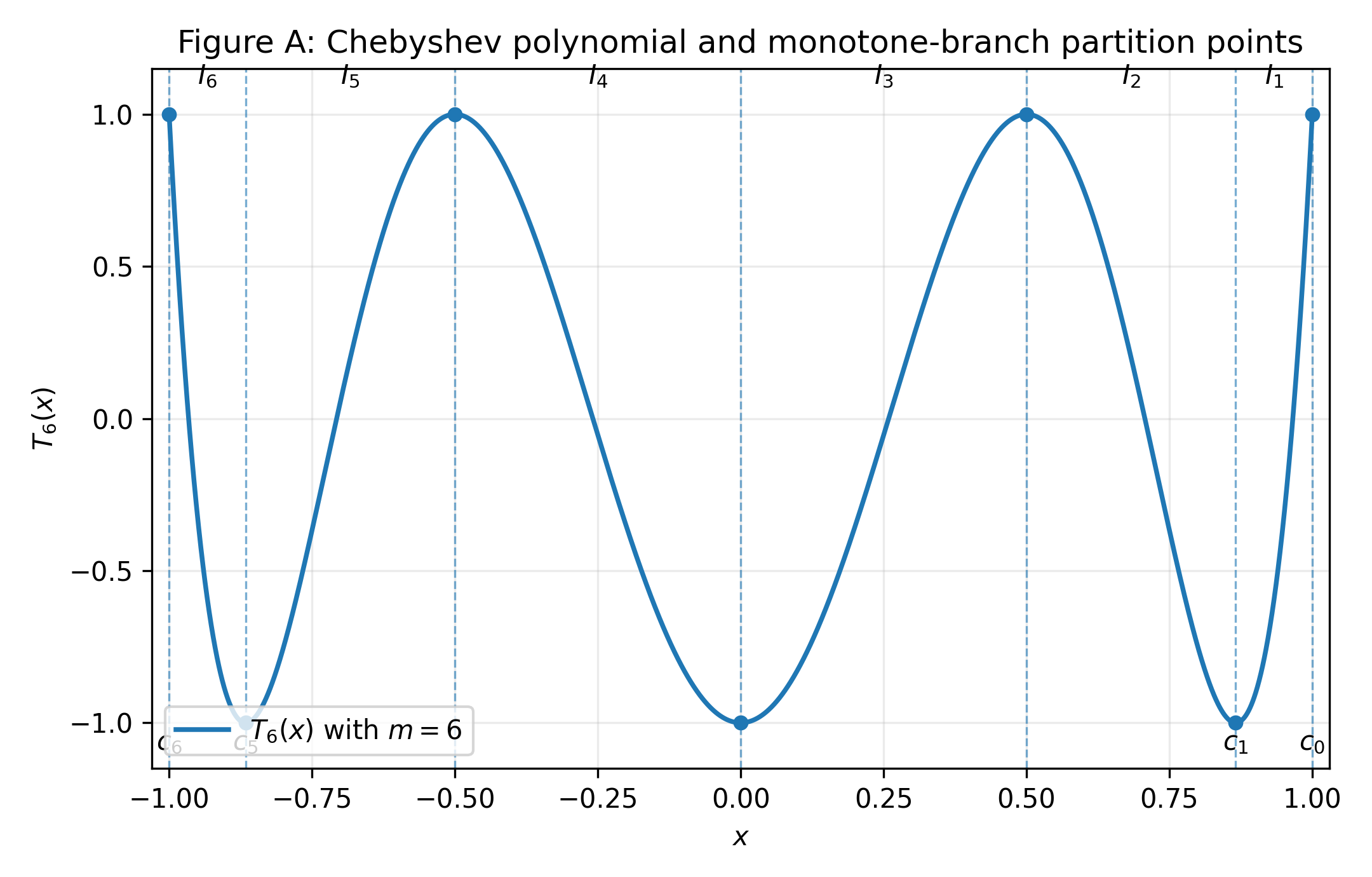}
  \caption{Chebyshev polynomial $T_m$ (here $m=6$). The points $c_k=\cos(k\pi/m)$ partition $[-1,1]$ into intervals
  $I_k=(c_k,c_{k-1})$ on which $T_m$ is strictly monotone and maps $I_k$ onto $(-1,1)$.}
  \label{fig:cheb-branches}
\end{figure}

\begin{proof}
On \((0,\pi)\), \(\theta\mapsto \cos\theta\) is strictly decreasing, so \(c_0>\cdots>c_m\).
Fix \(k\) and write \(x=\cos\theta\) with \(\theta\in(\frac{(k-1)\pi}{m},\frac{k\pi}{m})\), which corresponds exactly to \(x\in I_k\).
Then \(T_m(x)=\cos(m\theta)\) with \(m\theta\in((k-1)\pi,k\pi)\), where \(\cos\) is strictly monotone and has range \((-1,1)\) on that open interval. 
Therefore \(T_m\) is strictly monotone on \(I_k\) and maps it onto \((-1,1)\), giving a diffeomorphism and \(T_m'\neq 0\) on \(I_k\).
Figure~\ref{fig:cheb-branches} visualizes this decomposition: the points $c_k$ split $(-1,1)$ into the $m$ intervals
$I_k=(c_k,c_{k-1})$ on which $T_m$ is monotone and covers $(-1,1)$.

\end{proof}

\section{Replication via Chebyshev pullback}\label{sec:replication}

\begin{proof}[Proof of Theorem~\ref{thm:replication}.]
Let \(X(x,y)=(P(x,y),Q(x,y))\) be a planar polynomial vector field of degree \(\le n\) having limit cycles
\(\gamma_1,\dots,\gamma_k\).
Since \(\bigcup_{\ell=1}^k \gamma_\ell\) is compact, an invertible affine change of coordinates
(Lemma~\ref{lem:affine}) places all these cycles inside the square
\(S_\rho:=(-\rho,\rho)^2\) for some \(\rho\in(0,1)\).

Fix \(m\ge 2\) and let \(T_m\) be the Chebyshev polynomial of degree \(m\).
By Lemma~\ref{lem:chebbranches} there exist pairwise disjoint open intervals
\(I_1,\dots,I_m\subset(-1,1)\) such that each restriction \(T_m:I_i\to(-1,1)\) is a diffeomorphism and
\(T_m'(x)\neq 0\) on \(I_i\).
Define the separable map
\[
\Phi(u,v):=(T_m(u),T_m(v)).
\]
Then for every pair \((i,j)\), the restriction
\(\Phi:I_i\times I_j\to(-1,1)^2\) is a diffeomorphism, and \(T_m'(u)T_m'(v)\) has a constant, nonzero sign on
\(I_i\times I_j\).

\begin{figure}[t]
  \centering
  \includegraphics[width=0.95\linewidth]{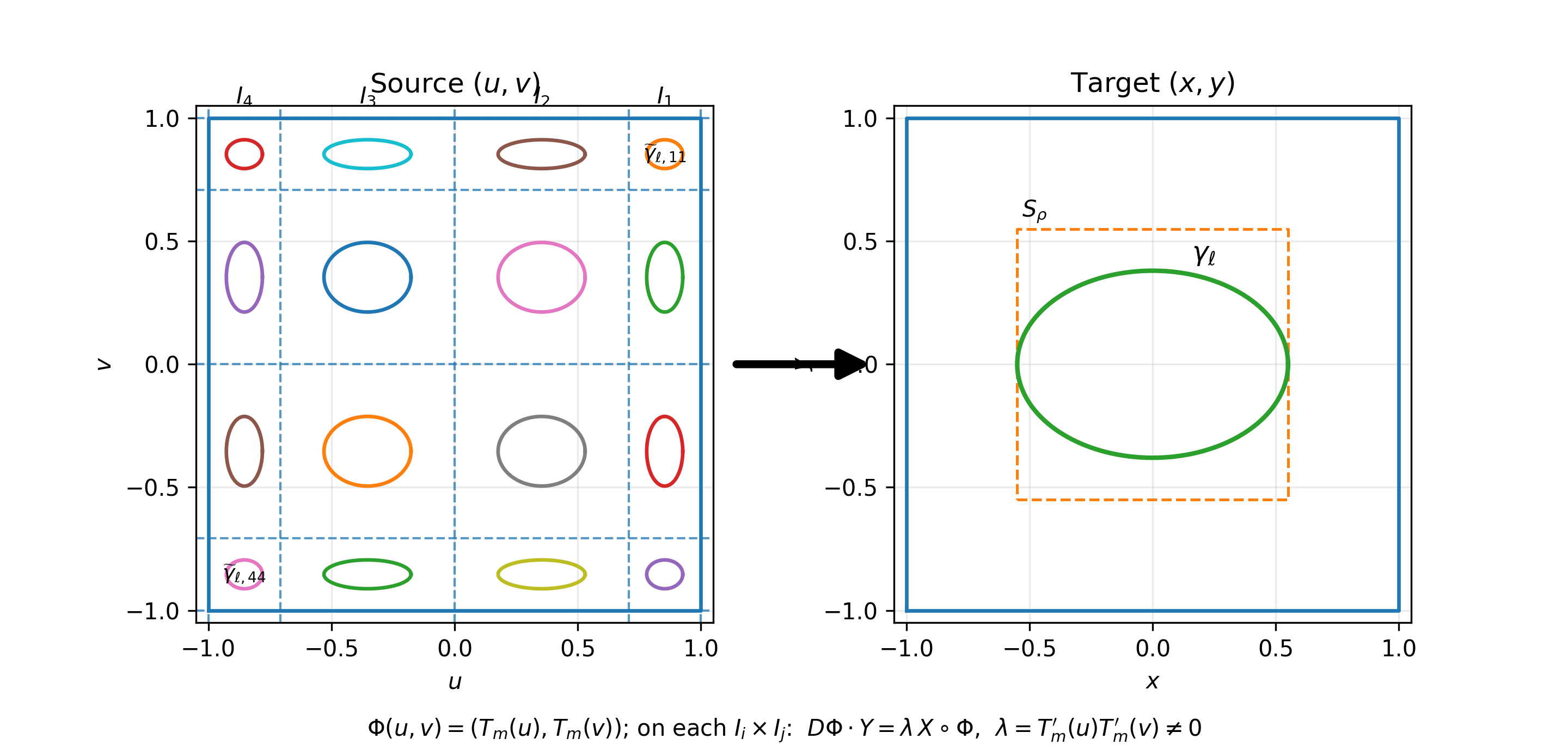}
  \caption{Replication schematic. A limit cycle $\gamma_\ell\subset S_\rho$ lifts under
  $\Phi(u,v)=(T_m(u),T_m(v))$ to disjoint cycles $\widetilde{\gamma}_{\ell,ij}$ inside the branch rectangles
  $I_i\times I_j$ (one shown per rectangle). On each rectangle,
  $D\Phi\cdot Y=\lambda\,X\circ\Phi$ with $\lambda=T_m'(u)T_m'(v)\neq 0$.}
  \label{fig:replication-schematic}
\end{figure}
Geometrically, the $m$ monotone full branches of $T_m$ partition $(-1,1)$ into disjoint intervals
$I_1,\dots,I_m$, so the domain $(u,v)$-plane is partitioned into $m^2$ disjoint branch rectangles
$I_i\times I_j$.  As illustrated in Figure~\ref{fig:replication-schematic}, on each such rectangle the map
$\Phi(u,v)=(T_m(u),T_m(v))$ is a diffeomorphism onto $(-1,1)^2$.  Consequently, any limit cycle
$\gamma_\ell\subset S_\rho$ has a unique lifted copy
$\widetilde{\gamma}_{\ell,ij}\subset I_i\times I_j$ in each branch rectangle, 
and these lifts are pairwise disjoint across different $(i,j)$.

\smallskip
\noindent\textbf{Pullback field and degree.}
Define a polynomial vector field \(Y\) on \(\mathbb{R}^2\) by
\[
\dot u = T_m'(v)\,P(T_m(u),T_m(v)),\qquad
\dot v = T_m'(u)\,Q(T_m(u),T_m(v)).
\]
Since \(\deg(T_m)=m\) and \(\deg(T_m')=m-1\), we have \(\deg(Y)\le nm+(m-1)=nm+m-1\).

\smallskip
\noindent\textbf{Orbit correspondence on each branch.}
On any rectangle \(I_i\times I_j\), the chain rule gives, for an orbit \((u(t),v(t))\subset I_i\times I_j\),
\begin{align*}
\frac{d}{dt}\Phi(u(t),v(t))
&=
\begin{pmatrix}T_m'(u) & 0\\ 0 & T_m'(v)\end{pmatrix}
\binom{\dot u}{\dot v} \\
&=
T_m'(u)T_m'(v)\,(P,Q)\bigl(\Phi(u(t),v(t))\bigr).
\end{align*}
Equivalently,
\[
D\Phi\cdot Y=\lambda\,X\circ\Phi
\quad\text{on } I_i\times I_j,\qquad
\lambda(u,v):=T_m'(u)T_m'(v)\neq 0.
\]
\[D\Phi\cdot Y=\lambda\,X\circ\Phi,\quad \lambda=T_m'(u)T_m'(v)\neq 0\]

\begin{figure}[htbp]
  \centering
  \includegraphics[width=0.5\linewidth]{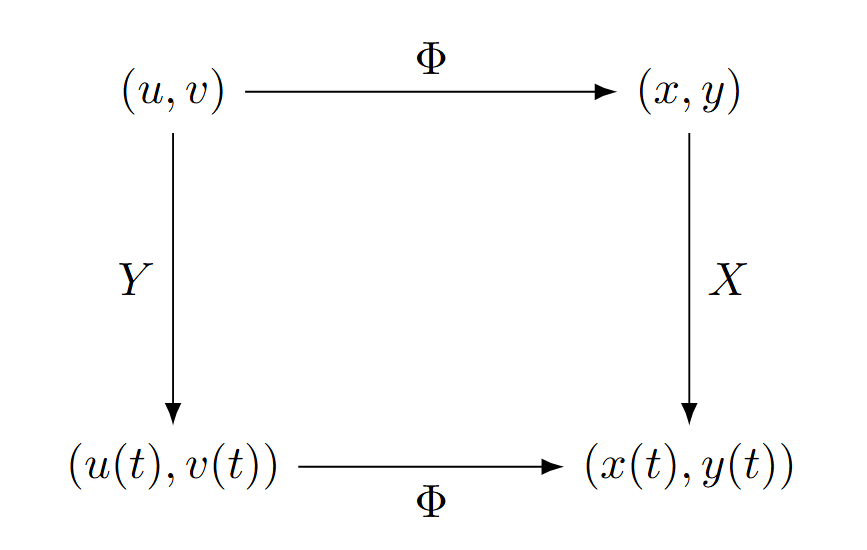}
  \caption{Commutative diagram for the conjugacy relation.}
  \label{fig:Lift}
\end{figure}

Thus \(\Phi\) sends \(Y\)-trajectories in \(I_i\times I_j\) onto \(X\)-trajectories in \((-1,1)^2\), with only a
nonvanishing time reparametrization (possibly reversing orientation if \(\lambda<0\)).
In particular, periodic orbits and their isolation properties transfer through \(\Phi\) on each branch
(cf.\ Lemma~\ref{lem:timechange}).

\smallskip
\noindent\textbf{Lifting limit cycles.}
Fix a limit cycle \(\gamma_\ell\subset S_\rho\).
For each \((i,j)\), define its lift
\[
\widetilde\gamma_{\ell,ij}:=\Phi^{-1}(\gamma_\ell)\cap(I_i\times I_j).
\]
Because \(\Phi|_{I_i\times I_j}\) is a diffeomorphism, \(\widetilde\gamma_{\ell,ij}\) is a simple closed curve and
\(\Phi(\widetilde\gamma_{\ell,ij})=\gamma_\ell\).
Moreover, since \(\gamma_\ell\Subset(-1,1)^2\) and \(\Phi|_{I_i\times I_j}\) is a homeomorphism, 
the lift \(\widetilde\gamma_{\ell,ij}\) is compactly contained in \(I_i\times I_j\); 
in particular it has positive distance from \(\partial(I_i\times I_j)\).
The orbit correspondence above implies that \(\widetilde\gamma_{\ell,ij}\) is \(Y\)-invariant and consists of
periodic trajectories, hence \(\widetilde\gamma_{\ell,ij}\) is a periodic orbit of \(Y\).
If \((i,j)\neq(i',j')\), then \(\widetilde\gamma_{\ell,ij}\) and \(\widetilde\gamma_{\ell,i'j'}\) are disjoint,
since the rectangles \(I_i\times I_j\) are disjoint.

Finally, 
\(\gamma_\ell\) is a \emph{limit} cycle of \(X\), so a local Poincar\'e return map near \(\gamma_\ell\) has an
isolated fixed point corresponding to \(\gamma_\ell\).
On each rectangle \(I_i\times I_j\), 
the dynamics of \(Y\) is conjugate to that of \(X\) via \(\Phi\) up to a nonvanishing time change; 
therefore the corresponding fixed point for
\(\widetilde\gamma_{\ell,ij}\) is isolated as well, and \(\widetilde\gamma_{\ell,ij}\) is a limit cycle of \(Y\).

\smallskip
\noindent\textbf{Counting.}
Each \(\gamma_\ell\) produces \(m^2\) distinct lifted limit cycles \(\widetilde\gamma_{\ell,ij}\), hence \(Y\) has at least
\(m^2k\) limit cycles.
Since \(\deg(Y)\le nm+m-1\), taking suprema over all degree-\(\le n\) fields \(X\) yields
\[
H(nm+m-1)\ \ge\ m^2\,H(n),
\]
as claimed.
\end{proof}

\section{Separable pullbacks and a quadratic bound}\label{sec:separable}

\subsection{Full-branch count for a univariate polynomial}

\begin{definition}[Full branches]\label{def:branches}
Let \(p\in\mathbb{R}[x]\) have degree \(m\ge 1\).
A \emph{full branch interval} for \(p\) is an open interval \(I\subset\mathbb{R}\) such that
\(p:I\to(-1,1)\) is a \(C^1\) diffeomorphism.
Let \(B(p)\) be the number of full branch intervals of \(p\).
\end{definition}

\begin{lemma}[Branch bound]\label{lem:branchbound}
If \(\deg(p)=m\), then \(B(p)\le m\).
\end{lemma}

\begin{proof}
The real line is partitioned into at most \(m\) open intervals on which \(p\) is monotone,
since \(p'\) has degree \(m-1\) and hence at most \(m-1\) real zeros.
On each monotonicity interval \(J\), the set \(p^{-1}((-1,1))\cap J\) is either empty or an open interval;
if it maps onto \((-1,1)\), it contributes at most one full branch. Hence \(B(p)\le m\).
\end{proof}

\begin{remark}
For \(p=T_m\), Lemma \ref{lem:chebbranches} gives \(B(T_m)=m\).
\end{remark}

\subsection{General separable replication}

\begin{proposition}[Separable replication]\label{prop:generalrep}
Let \(X(x,y)=(P(x,y),Q(x,y))\) be a planar polynomial vector field and
assume all its limit cycles lie in some compact \(K\Subset(-1,1)^2\).
Let \(p\in\mathbb{R}[x]\) have degree \(m\ge 2\) and full branch intervals \(I_1,\dots,I_{B(p)}\).
Define
\[
\Phi(u,v)=(p(u),p(v)),
\]
and define a polynomial vector field \(Y\) by
\[
\dot u = p'(v)\,P(p(u),p(v)),\qquad
\dot v = p'(u)\,Q(p(u),p(v)).
\]
Then:
\begin{enumerate}
\item \(\deg(Y)\le m\,\deg(X)+(m-1)\).
\item For each rectangle \(I_a\times I_b\), \(\Phi\) restricts to a diffeomorphism onto \((-1,1)^2\), and on \(I_a\times I_b\) one has
\[
D\Phi\cdot Y = (p'(u)p'(v))\, X\circ \Phi,
\]
with \(p'(u)p'(v)\neq 0\). Hence \(\Phi\) maps \(Y\)-orbits to \(X\)-orbits and
the local Poincar\'e return maps are conjugate.
\item Each limit cycle of \(X\) lifts to at least \(B(p)^2\) pairwise disjoint limit cycles of \(Y\).
\end{enumerate}
Consequently, for every \(n\ge 1\) and every polynomial \(p\) of degree \(m\ge 2\),
\[
H(nm+m-1)\ \ge\ B(p)^2\,H(n).
\]
\end{proposition}

\begin{proof}
(1) Immediate from \(\deg(p)=m\), \(\deg(p')=m-1\), and degree multiplication under composition.

(2) On \(I_a\times I_b\), \(p'\neq 0\) and \(p\) is a diffeomorphism onto \((-1,1)\) in each coordinate, so \(\Phi\) is a diffeomorphism.
For an orbit \((u(t),v(t))\subset I_a\times I_b\), letting \((x,y)=\Phi(u,v)\) yields
\[
\dot x=p'(u)\dot u=p'(u)p'(v)\,P(x,y),\qquad
\dot y=p'(v)\dot v=p'(u)p'(v)\,Q(x,y),
\]
which is the displayed identity. 
Apply Lemma \ref{lem:timechange} to obtain orbit correspondence and conjugacy of return maps.

(3) For each limit cycle \(\gamma\subset K\), define its lifts \(\Phi^{-1}(\gamma)\cap(I_a\times I_b)\).
These are periodic orbits by (2).
Disjointness follows from disjoint rectangles.
Because \(K\Subset(-1,1)^2\) and \(\Phi|_{I_a\times I_b}\) is a diffeomorphism,
each lift lies a positive distance from the rectangle boundary;
the return maps are conjugate by (2), so isolation transfers.
Hence each lift is a limit cycle.
\end{proof}

\subsection{Exact degree growth for the pullback construction}

\begin{lemma}[Exact degree growth]\label{lem:exactdegree}
Let \(X=(P,Q)\) be a planar polynomial vector field of degree \(d:=\deg(X)\ge 1\), and
let \(p\) be a polynomial of degree \(m\ge 2\).
Form \(Y\) by the separable pullback construction of Proposition \ref{prop:generalrep}.
Then
\[
\deg(Y)= m\,d + (m-1),\qquad\text{equivalently}\qquad \deg(Y)+1 = m(\deg(X)+1).
\]
\end{lemma}

\begin{proof}
Write \(p(u)=cu^m+\text{(lower degree)}\) with \(c\neq 0\).
Let \(P_d,Q_d\) be the homogeneous parts of total degree \(d\). 
At least one of \(P_d,Q_d\) is nonzero.

If \(P_d\not\equiv 0\), then \(P(p(u),p(v))\) has top-degree part \(P_d(cu^m,cv^m)\), a nonzero homogeneous polynomial of total degree \(md\). 
Multiplying by \(p'(v)=mc\,v^{m-1}+\text{(lower)}\) produces a nonzero contribution of total degree \(md+(m-1)\) in \(\dot u\), and
no lower-degree terms can cancel it.
Hence \(\deg(\dot u)=md+(m-1)\).

If instead \(P_d\equiv 0\) and \(Q_d\not\equiv 0\), the same argument applies to \(\dot v=p'(u)Q(p(u),p(v))\), giving \(\deg(\dot v)=md+(m-1)\).

Therefore \(\deg(Y)=\max\{\deg(\dot u),\deg(\dot v)\}=md+(m-1)\).
\end{proof}

\subsection{Quadratic ceiling for \emph{pure} separable replication}

\begin{proof}[Proof of Theorem~\ref{thm:quadraticceiling}.]
At step \(j\), the number of carried-over limit cycles multiplies by at most \(B(p_j)^2\le m_j^2\) (Lemma \ref{lem:branchbound}).
Thus
\[
\pi(X_r)\ \le\ k_0\prod_{j=1}^r m_j^2.
\]
By Lemma \ref{lem:exactdegree}, each pullback step satisfies
\[
\deg(X_j)+1 = m_j(\deg(X_{j-1})+1).
\]
Therefore
\begin{align*}
N+1=\deg(X_r)+1 &= (\deg(X_0)+1)\prod_{j=1}^r m_j,\\
\text{so}\qquad
\prod_{j=1}^r m_j &= \frac{N+1}{\deg(X_0)+1}.
\end{align*}
Substitute into the cycle bound to get
\[
\pi(X_r)\ \le\ k_0\left(\frac{N+1}{\deg(X_0)+1}\right)^2
\ =\ k_0\left(\frac{N+1}{n_0+1}\right)^2,
\]
as claimed.
\end{proof}

\begin{remark}[Why superquadratic lower bounds need more than replication]
Theorem \ref{thm:quadraticceiling} isolates the obstruction:
\emph{replication-only} via separable pullbacks cannot yield superquadratic growth in degree.
Thus any \(O(n^2\log n)\)-type lower bound must incorporate additional cycle-creation mechanisms beyond mere lifting 
(see, e.g., 
\cite{HanLi2012} 
and the recent summary discussions around Hilbert-number growth such as
\cite{GasullSantanaCPAA2026,BuzziNovaes2024}).
\end{remark}

\begin{remark}[Open direction: non-separable polynomial coverings]\label{rem:nonseparable}
The quadratic ceiling in Theorem~\ref{thm:quadraticceiling} is proved for \emph{separable} pullbacks
\(\Phi(u,v)=(p(u),p(v))\), where the replication factor is controlled by the one-dimensional full-branch count \(B(p)\).
It is natural to ask what changes if one allows \emph{non-separable} polynomial maps
\[
\Phi(u,v)=(p(u,v),\,q(u,v)).
\]
Whenever \(\det D\Phi(u,v)\neq 0\), one can define a polynomial ``pullback'' vector field by
\[
Y(u,v):=\operatorname{adj}(D\Phi(u,v))\,X(\Phi(u,v)),
\]
i.e.
\[
\dot u = q_v\,P(\Phi)-p_v\,Q(\Phi),\qquad
\dot v = -q_u\,P(\Phi)+p_u\,Q(\Phi),
\]
so that
\[
D\Phi\cdot Y = (\det D\Phi)\,X\circ\Phi.
\]
Thus, on any region where \(\Phi\) restricts to a diffeomorphism and \(\det D\Phi\) has constant sign, 
the orbit sets
(and local Poincar\'e return maps, up to inversion) 
correspond exactly as in the separable case.

The open problem is then to quantify replication in terms of the \emph{two-dimensional branch geometry} of \(\Phi\):
how many pairwise disjoint domains can \(\mathbb{R}^2\) be partitioned into on which \(\Phi\) is a diffeomorphism onto a
fixed target region such as \((-1,1)^2\), and 
how does this number depend on \(\deg(p)\) and \(\deg(q)\)?
In particular,
it would be interesting to determine whether an analogue of Theorem~\ref{thm:quadraticceiling} holds for
replication-only schemes built from such non-separable coverings,
or whether genuinely new growth rates become possible.
\end{remark}

\section[A worked example with m=3]{A worked example: \(m=3\) yields nine limit cycles from one}\label{sec:workedexample}
This section illustrates Theorem~\ref{thm:replication} on an explicit polynomial vector field with a single hyperbolic limit cycle.
We then apply the Chebyshev pullback with \(m=3\) to obtain \(3^2=9\) disjoint hyperbolic
limit cycles in the \(3\times 3\) branch rectangles.

\begin{figure}[htbp]
\centering
\begin{minipage}[t]{0.49\linewidth}
  \centering
  \includegraphics[width=\linewidth]{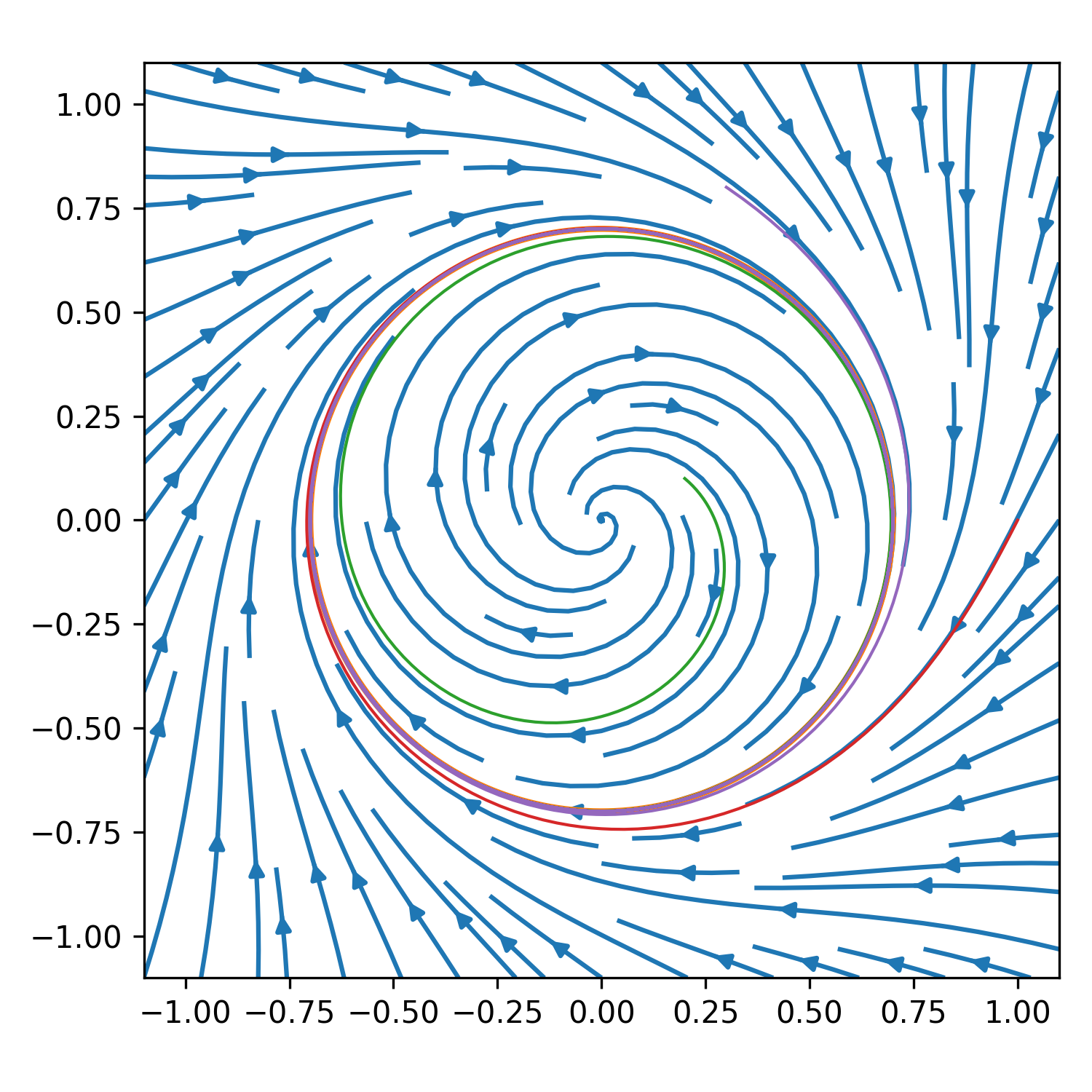}
  \smallskip

  \small\textbf{(a)} Original system \eqref{eq:radial-cubic-rho}:
  the hyperbolic limit cycle \(\gamma_\rho=\{x^2+y^2=\rho^2\}\).
\end{minipage}\hfill
\begin{minipage}[t]{0.49\linewidth}
  \centering
  \includegraphics[width=\linewidth]{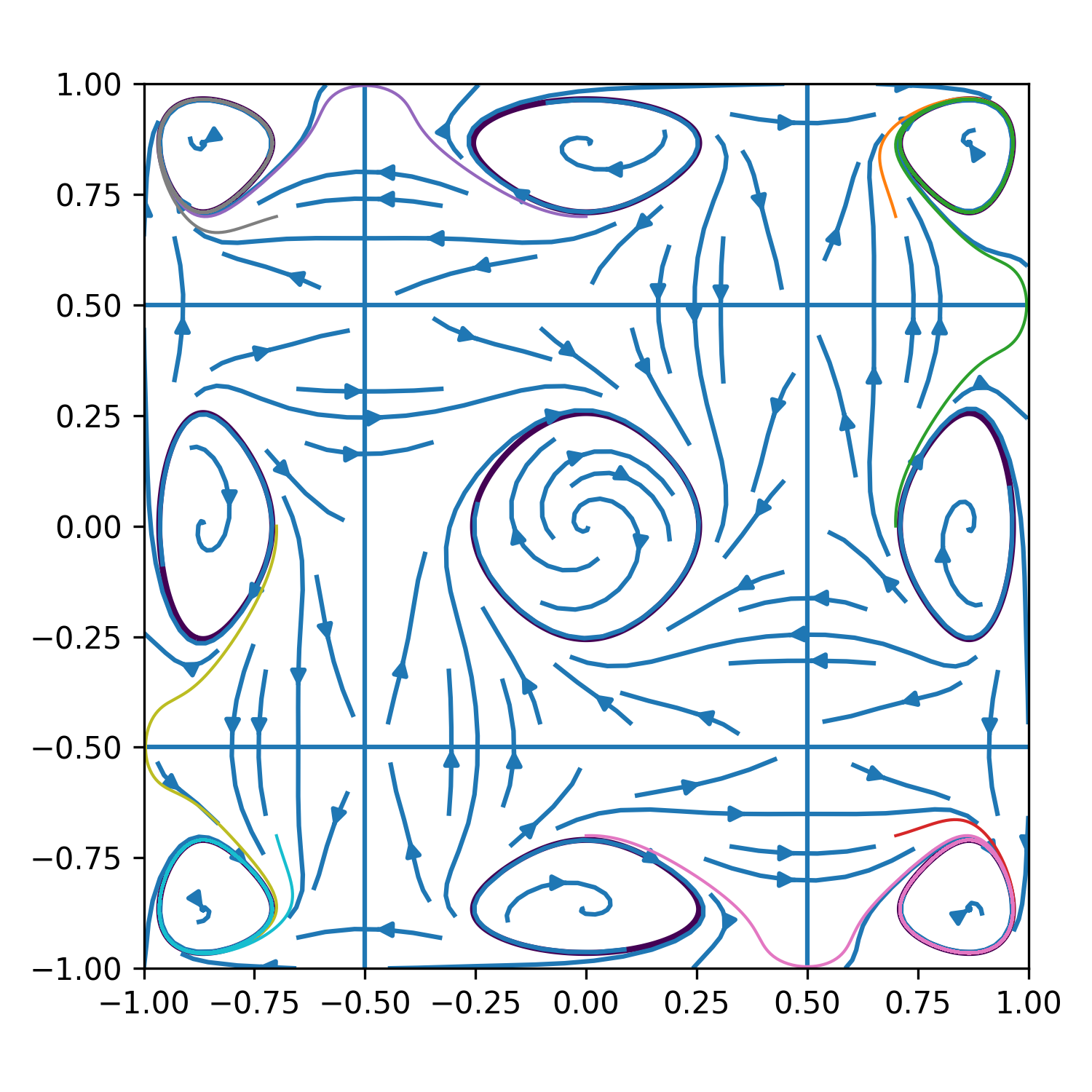}
  \smallskip

  \small\textbf{(b)} Pullback with \(m=3\):
  the lifted cycles \(\Phi^{-1}(\gamma_\rho)\) split into \(9\) components,
  one in each branch rectangle \(I_i\times I_j\).
\end{minipage}
\caption{Visual illustration of the worked example in Section~\ref{sec:workedexample}.}
\label{fig:workedexample-plots}
\end{figure}

\medskip
\noindent\textbf{Visual illustration.}
Although the replication argument is entirely analytic, it is helpful to visualize the geometry.
Figure~\ref{fig:workedexample-plots}(a) shows the phase portrait of \eqref{eq:radial-cubic-rho} and the limit cycle
\(\gamma_\rho=\{x^2+y^2=\rho^2\}\).
Figure~\ref{fig:workedexample-plots}(b) depicts the pullback geometry for \(m=3\): the lifted set is given exactly by the
implicit equation
\[
T_3(u)^2+T_3(v)^2=\rho^2,
\]
which decomposes into nine disjoint closed components,
one contained in each branch rectangle \(I_i\times I_j\).
The branch boundaries correspond to the critical lines \(u=\pm\frac12\) and \(v=\pm\frac12\), where \(T_3'=0\).

\subsection{A cubic system with one hyperbolic limit cycle}

Fix \(\rho\in(0,1)\) and let \(X_\rho=(P_\rho,Q_\rho)\) be the cubic polynomial vector field defined by
\begin{equation}\label{eq:radial-cubic-rho}
\dot x = P_\rho(x,y)= y - x\bigl(x^2+y^2-\rho^2\bigr),\qquad
\dot y = Q_\rho(x,y)= -x - y\bigl(x^2+y^2-\rho^2\bigr).
\end{equation}
Write \(x=r\cos\theta\), \(y=r\sin\theta\). Since
\[
(\dot x,\dot y)=(y,-x) - (r^2-\rho^2)(x,y),
\]
the first term is a rigid rotation and the second term is purely radial.
A direct computation gives
\begin{align*}
\dot r &= \frac{x\dot x+y\dot y}{r} \\
       &= \frac{x\bigl(y-x(r^2-\rho^2)\bigr)+y\bigl(-x-y(r^2-\rho^2)\bigr)}{r} \\
       &= -r(r^2-\rho^2)=r(\rho^2-r^2).
\end{align*}
and
\[
\dot\theta=\frac{x\dot y-y\dot x}{r^2}=-1.
\]
Hence \(r=\rho\) is a periodic orbit. 
It is \emph{hyperbolic} because for \(f(r):=r(\rho^2-r^2)\) one has
\(f'(\rho)=\rho^2-3\rho^2=-2\rho^2\neq 0\).
Moreover, 
\(\dot r>0\) for \(0<r<\rho\) and \(\dot r<0\) for \(r>\rho\), so the periodic orbit is isolated and therefore a
(limit) cycle. 
Denote it by
\[
\gamma_\rho:=\{(x,y):x^2+y^2=\rho^2\}\subset (-\rho,\rho)^2\Subset(-1,1)^2.
\]
Thus this explicit example shows \(H(3)\ge 1\).

\subsection[Chebyshev pullback with m=3]{Chebyshev pullback with \(m=3\)}

Recall
\[
T_3(t)=4t^3-3t,\qquad T_3'(t)=12t^2-3.
\]
The critical points of \(T_3\) in \((-1,1)\) are \(\pm\frac12\), so the three full-branch intervals can be taken as
\[
I_1=\Bigl(\frac12,1\Bigr),\qquad
I_2=\Bigl(-\frac12,\frac12\Bigr),\qquad
I_3=\Bigl(-1,-\frac12\Bigr),
\]
on each of which \(T_3:I_k\to(-1,1)\) is a diffeomorphism (and \(T_3'\neq 0\) on \(I_k\)).
Define the separable map
\[
\Phi(u,v)=(T_3(u),T_3(v)).
\]
For each \((i,j)\in\{1,2,3\}^2\), the restriction \(\Phi:I_i\times I_j\to(-1,1)^2\) is a diffeomorphism.

Define the pullback vector field \(Y\) by the Chebyshev construction (cf.\ Theorem~\ref{thm:replication}):
\begin{align}\label{eq:pullback-Y-m3}
\dot u
&= T_3'(v)\,P_\rho(T_3(u),T_3(v)) \\
&= T_3'(v)\Bigl(T_3(v)
   - T_3(u)\bigl(T_3(u)^2+T_3(v)^2-\rho^2\bigr)\Bigr),\nonumber\\
\dot v
&= T_3'(u)\,Q_\rho(T_3(u),T_3(v)) \\
&= T_3'(u)\Bigl(-T_3(u)
   - T_3(v)\bigl(T_3(u)^2+T_3(v)^2-\rho^2\bigr)\Bigr).\nonumber
\end{align}
Since \(\deg(X_\rho)=3\), \(\deg(T_3)=3\), and \(\deg(T_3')=2\), the degree bookkeeping gives
\[
\deg(Y)=3\cdot 3+(3-1)=11.
\]

\subsection{Nine disjoint lifted hyperbolic limit cycles}

Pointwise on each branch rectangle \(I_i\times I_j\), the chain rule yields
\[
D\Phi\cdot Y = \lambda\,X_\rho\circ\Phi,\qquad
\lambda(u,v)=T_3'(u)T_3'(v)\neq 0 \ \text{on } I_i\times I_j,
\]
so \(\Phi\) maps \(Y\)-trajectories in \(I_i\times I_j\) onto \(X_\rho\)-trajectories in \((-1,1)^2\) up to a
nonvanishing time change (Lemma~\ref{lem:timechange}). 
Consequently, local Poincar\'e return maps are conjugate
(up to inversion when \(\lambda<0\)).

For each \((i,j)\), define
\[
\widetilde\gamma_{ij}:=\Phi^{-1}(\gamma_\rho)\cap(I_i\times I_j).
\]
Because \(\Phi|_{I_i\times I_j}\) is a diffeomorphism, \(\widetilde\gamma_{ij}\) is a simple closed curve and
\(\Phi(\widetilde\gamma_{ij})=\gamma_\rho\). Since \(\gamma_\rho\Subset(-1,1)^2\), each lift
\(\widetilde\gamma_{ij}\) is compactly contained in \(I_i\times I_j\), 
hence stays a positive distance from the rectangle boundary.
The orbit correspondence implies \(\widetilde\gamma_{ij}\) is a periodic orbit of \(Y\).

Furthermore, \(\gamma_\rho\) is hyperbolic, 
so the corresponding return map has multiplier \(\mu\neq 1\).
Under conjugacy the multiplier is preserved, and 
under inversion it becomes \(\mu^{-1}\neq 1\); 
hence each \(\widetilde\gamma_{ij}\) is a hyperbolic limit cycle of \(Y\).
The rectangles \(I_i\times I_j\) are pairwise disjoint, so the nine cycles \(\widetilde\gamma_{ij}\) are pairwise disjoint.

 Therefore \(Y\) is a polynomial vector field of degree \(11\) with at least \(9\) limit cycles, and hence
\[
H(11)\ge 9.
\]

\section{Conclusion}\label{sec:conclusion}

We have revisited one of the basic amplification mechanisms in the study of Hilbert numbers: replication of limit cycles through polynomial pullbacks.
Our first main result is a clean Chebyshev-based replication theorem.
For every \(m\ge 2\), the separable map
\[
\Phi(u,v)=(T_m(u),T_m(v))
\]
has \(m\) full monotone branches in each coordinate, hence \(m^2\) branch rectangles on which \(\Phi\) is a diffeomorphism onto \((-1,1)^2\).
Pulling back a degree-\(\le n\) polynomial vector field through this map and using invariance under nonvanishing time reparametrization yields
\[
H(nm+m-1)\ge m^2H(n).
\]
This gives a fully explicit and self-contained replication step with exact degree growth
\[
\deg(Y)=m\,\deg(X)+(m-1).
\]

Our second main result identifies a structural limitation of \emph{pure} separable replication.
After abstracting the construction to general pullbacks of the form \((u,v)\mapsto (p(u),p(v))\), we proved that the number of full branches of a degree-\(m\) polynomial \(p\) is at most \(m\), with equality attained by the Chebyshev polynomial.
As a consequence, if one iterates separable pullbacks and assumes that no new limit cycles are created beyond the lifted copies forced by branch geometry, then the total number of limit cycles in the resulting degree-\(N\) system is bounded by
\[
\pi(X)\le k_0\left(\frac{N+1}{n_0+1}\right)^2.
\]
Thus replication-only schemes of this type can produce at most quadratic growth in the degree.
In particular, the stronger \(N^2\log N\)-type lower bounds in the literature must rely on additional cycle-creation mechanisms beyond separable lifting; see, for example, \cite{ChristopherLloyd1995,LiChanChung2002,HanLi2012}.

A further point of the paper is that the Chebyshev replication theorem is not only structurally sharp within the separable class, but also numerically competitive when combined with strong seed bounds from the literature.
Using the currently best published values recorded by Han--Li and Prohens--Torregrosa, our one-step replication bound matches several published records and improves others.
In particular, the comparison carried out in Section~\ref{sec:introduction} and Appendix~\ref{app:derivations} yields the explicit estimates
\[
H(14)\ge 252,\qquad
H(29)\ge 1080,\qquad
H(31)\ge 1380,\qquad
H(39)\ge 2012,
\]
which improve the previously published bounds in those degrees.
This shows that even though pure replication cannot explain the known superquadratic asymptotic growth, it remains a useful degree-amplification tool at the concrete numerical level.

Several directions remain open.
The most natural one is to move beyond separable coverings and investigate polynomial pullbacks of the form
\[
\Phi(u,v)=(p(u,v),q(u,v)),
\]
where the relevant branch geometry is genuinely two-dimensional.
It would be very interesting to understand whether analogues of our quadratic ceiling persist in that broader setting, or whether non-separable coverings allow genuinely different replication rates.
More generally, it remains an important problem to determine how replication mechanisms can be combined with local bifurcation mechanisms in a systematic way so as to obtain stronger explicit lower bounds for \(H(n)\) while keeping degree growth under control.

\appendix
\section{Derivation of the lower bounds used in Table~\ref{tab:pub-vs-cheb}}
\label{app:derivations}

In this appendix we record the seed bounds taken from the literature and explain how
the numbers in Table~\ref{tab:pub-vs-cheb} are obtained.

Let \(L_{\mathrm{pub}}(n)\) denote the best currently published lower bound used as a seed.
For the degrees needed in Table~\ref{tab:pub-vs-cheb}, we take:
\begin{align*}
L_{\mathrm{pub}}(4) &= 28, & 
L_{\mathrm{pub}}(5) &= 37, & 
L_{\mathrm{pub}}(6) &= 53, \\
L_{\mathrm{pub}}(7) &= 74, & 
L_{\mathrm{pub}}(8) &= 96, & 
L_{\mathrm{pub}}(9) &= 120, \\
L_{\mathrm{pub}}(10) &= 142,
\end{align*}
from Prohens--Torregrosa \cite[Thm.~1]{ProhensTorregrosa2019},
together with
\begin{align*}
L_{\mathrm{pub}}(11) &= 153, & 
L_{\mathrm{pub}}(12) &= 157, & 
L_{\mathrm{pub}}(14) &= 194, \\
L_{\mathrm{pub}}(15) &= 345, & 
L_{\mathrm{pub}}(16) &= 351, & 
L_{\mathrm{pub}}(18) &= 372, \\
L_{\mathrm{pub}}(19) &= 503, & 
L_{\mathrm{pub}}(20) &= 509,
\end{align*}
from Han--Li \cite[Thm.~1.2(i)]{HanLi2012}.
For the degrees where Prohens--Torregrosa improve Han--Li, we use
\begin{align*}
L_{\mathrm{pub}}(13) &= 212, & 
L_{\mathrm{pub}}(17) &= 384, & 
L_{\mathrm{pub}}(21) &= 568, \\
L_{\mathrm{pub}}(31) &= 1184, & 
L_{\mathrm{pub}}(35) &= 1536, & 
L_{\mathrm{pub}}(39) &= 1920, \\
L_{\mathrm{pub}}(43) &= 2272.
\end{align*}
from \cite[Cor.~2(a)]{ProhensTorregrosa2019}.

For a target degree \(N\), Theorem~\ref{thm:replication} gives
\[
H((n+1)m-1)\ge m^{2}H(n).
\]
Hence, whenever \(N+1=(n+1)m\) with \(m\ge2\), we obtain
\[
H(N)\ge m^{2}L_{\mathrm{pub}}(n).
\]
The value \(L_{\mathrm{Ch}}(N)\) in Table~\ref{tab:pub-vs-cheb} is simply the maximum of
these direct one-step consequences over all admissible factorizations \(N+1=(n+1)m\).

\begin{table}[htbp]
\centering
\small
\setlength{\tabcolsep}{5pt}
\resizebox{\textwidth}{!}{%
\begin{tabular}{ccccc}
\toprule
\(N\) & factorization of \(N+1\) & seed bound used & theorem output & value of \(L_{\mathrm{Ch}}(N)\)\\
\midrule
11 & \(12=6\cdot 2\)   & \(H(5)\ge 37\)   & \(H(11)\ge 4\cdot 37\)   & \(148\)  \\
13 & \(14=7\cdot 2\)   & \(H(6)\ge 53\)   & \(H(13)\ge 4\cdot 53\)   & \(212\)  \\
14 & \(15=5\cdot 3\)   & \(H(4)\ge 28\)   & \(H(14)\ge 9\cdot 28\)   & \(252\)  \\
15 & \(16=8\cdot 2\)   & \(H(7)\ge 74\)   & \(H(15)\ge 4\cdot 74\)   & \(296\)  \\
17 & \(18=9\cdot 2\)   & \(H(8)\ge 96\)   & \(H(17)\ge 4\cdot 96\)   & \(384\)  \\
19 & \(20=10\cdot 2\)  & \(H(9)\ge 120\)  & \(H(19)\ge 4\cdot 120\)  & \(480\)  \\
20 & \(21=7\cdot 3\)   & \(H(6)\ge 53\)   & \(H(20)\ge 9\cdot 53\)   & \(477\)  \\
21 & \(22=11\cdot 2\)  & \(H(10)\ge 142\) & \(H(21)\ge 4\cdot 142\)  & \(568\)  \\
23 & \(24=8\cdot 3\)   & \(H(7)\ge 74\)   & \(H(23)\ge 9\cdot 74\)   & \(666\)  \\
24 & \(25=5\cdot 5\)   & \(H(4)\ge 28\)   & \(H(24)\ge 25\cdot 28\)  & \(700\)  \\
25 & \(26=13\cdot 2\)  & \(H(12)\ge 157\) & \(H(25)\ge 4\cdot 157\)  & \(628\)  \\
26 & \(27=9\cdot 3\)   & \(H(8)\ge 96\)   & \(H(26)\ge 9\cdot 96\)   & \(864\)  \\
27 & \(28=14\cdot 2\)  & \(H(13)\ge 212\) & \(H(27)\ge 4\cdot 212\)  & \(848\)  \\
29 & \(30=10\cdot 3\)  & \(H(9)\ge 120\)  & \(H(29)\ge 9\cdot 120\)  & \(1080\) \\
31 & \(32=16\cdot 2\)  & \(H(15)\ge 345\) & \(H(31)\ge 4\cdot 345\)  & \(1380\) \\
35 & \(36=18\cdot 2\)  & \(H(17)\ge 384\) & \(H(35)\ge 4\cdot 384\)  & \(1536\) \\
39 & \(40=20\cdot 2\)  & \(H(19)\ge 503\) & \(H(39)\ge 4\cdot 503\)  & \(2012\) \\
43 & \(44=22\cdot 2\)  & \(H(21)\ge 568\) & \(H(43)\ge 4\cdot 568\)  & \(2272\) \\
\bottomrule
\end{tabular}%
}
\caption{Direct derivation of the values \(L_{\mathrm{Ch}}(N)\) from
Theorem~\ref{thm:replication} and the seed bounds available in the literature.}
\label{tab:cheb-derivation}
\end{table}

For example,
\[
30=(9+1)\cdot 3,
\]
so Theorem~\ref{thm:replication} with \(n=9\) and \(m=3\) yields
\[
H(29)\ge 9H(9)\ge 9\cdot 120=1080.
\]
Likewise,
\[
32=(15+1)\cdot 2
\]
gives
\[
H(31)\ge 4H(15)\ge 4\cdot 345=1380,
\]
and
\[
40=(19+1)\cdot 2
\]
gives
\[
H(39)\ge 4H(19)\ge 4\cdot 503=2012.
\]
All other entries are obtained in exactly the same way.

\end{document}